\documentclass[12pt]{article}
\usepackage{xypic}
\usepackage{epsfig}
\usepackage{tikz}
\usepackage{graphicx}
\usepackage{amsthm}
\usepackage{amssymb}
\usepackage{caption}
\usepackage{subcaption}
\usepackage{color}
\usepackage{hyperref}
\usepackage{amsmath,mathrsfs,amsfonts,verbatim,enumitem,leftidx}
\usepackage{enumitem}
\usepackage{cite}

\newtheorem{prop}{Proposition}[section]
\newtheorem{thm}{Theorem}[section]
\newtheorem{lem}{Lemma}[section]

\numberwithin{equation}{section}
\newtheorem{claim}{Claim}

\newcommand{\mc}[1]{\ensuremath{\mathcal{#1}}}
\newcommand{\mb}[1]{\ensuremath{\mathbb{#1}}}

\newcommand*{\sep}{ \ensuremath{ {\rm sep} }}
\newcommand*{\umdim}{ \ensuremath{ {\rm\overline{mdim}_M}} }
\newcommand*{\lmdim}{ \ensuremath{ {\rm\underline{mdim}_M} } }

\begin{document}

\title{Variational  principles of metric mean dimension for random dynamical systems
 \footnotetext {* Corresponding author}
  \footnotetext {2010 Mathematics Subject Classification: }}
\author{ Yunping Wang$^{1}$, Ercai Chen$^{2}$ and  Kexiang Yang$^{3*}$ \\ 
	\small 1 School of Science, Ningbo University of Technology,\\
	\small  Ningbo 315211, Zhejiang, P.R.China\\
	\small 2 School of Mathematical Sciences and Institute of Mathematics, Nanjing Normal University, \\ 
	\small Nanjing 210046, Jiangsu, P.R.China \\
	\small 3 School of Mathematical Sciences, Fudan University,  \\ \small Shanghai 200433, Shanghai, P. R. China \\
  \small    e-mail: yunpingwangj@126.com, ecchen@njnu.edu.cn, kxyangs@163.com 
}
\date{}
\maketitle

\begin{center}
 \begin{minipage}{120mm}
{\small {\bf Abstract.}
It  is well-known that the relativized  variational principle established by Bogenschutz and Kifer  connects the  fiber topological entropy and fiber measure-theoretic entropy. In context of random dynamical systems,  metric mean dimension was introduced to characterize infinite fiber entropy systems.  We  give four types of measure-theoretic $\epsilon$-entropies, called  measure-theoretic entropy of partitions decreasing in diameter, Shapira's entropy, Katok's entropy and Brin-Katok local entropy, and   establish four variational principles for metric mean dimension.}
\end{minipage}
 \end{center}

\vskip0.5cm
\section{Introduction}

~~~~The concepts of entropy plays a vital role in topological dynamics. The fundamental two  kinds of entropies which  receive  a lot of attention are the topological entropy and measure-theoretic entropy, which are related by the  well-known variational principle established by Goodwyn \cite{Goody}
 and Goodman  \cite{Goodman}:
 \begin{align*}
 h_{top}(T)=\sup_{\mu \in M(X,T)}h_{\mu}(T),
 \end{align*}
 where $T$ is a homeomorphism from a compact metric space $X$ to itself and  the supremum is take over all  $T$-invariant  Borel  probability measures  on $X$. 
 
Mean dimension, firstly introduced by Gromov \cite{Gro}, has been studied as a topological invariant of dynamical systems, which  quantifies the complexity of dynamical systems of infinite entropy.
Its applications especially in embedding problems have been stated in \cite{LT,LWE, GLT16,GYM}. In  \cite{LWE}, Lindenstrauss and Weiss introduced  the   dynamical version of Minkowski dimension, which is known as  metric mean dimension.  Inspired by the  variational principle for  topological entropy, it is expected to establish variational principles for  metric mean dimension in the setting of infinite entropy systems.  The lacking of the role of measure-theoretic metric mean dimension is the  main obstruction. In 2018, using the foundation of loss data compression method Lindenstrauss and Tsukamoto \cite{LT18}  established the first  variational principles  for  metric mean dimension  in terms of rate distortion functions. Besides,
a double variational principle is established for mean dimension \cite{LT19}, and is generalized to mean dimension with potentials  by Tsukamoto\cite{Tsu}.  By  replacing the rate distortion functions with other measure-theoretic $\epsilon$-entropies,  Gutman and  Spiewak \cite{Gut} derived a variational principle for metric mean dimension  involving  growth rates of measure-theoretic entropy of partitions decreasing  in diametric.  Shi \cite{Shi} obtained the  variational principles  for metric mean dimension in terms of  Shapira's entropy related to finite open covers, Katok's entropy and Brin-Katok local entropy. 
 
The present paper focus on  continuous bundle random dynamical system $T=(T_{\omega})_{\omega}$ over a measure-preserving system $(\Omega,\mc{F},\mb{P},\theta)$.  $\Theta:\Omega\times X \rightarrow \Omega \times X$ is the induced   skew product transformation.  Bogenschutz \cite{Bog} and Kifer \cite{kfier2001}  established the following  classical variational principle for random dynamical systems:
\begin{align*}
h_{top}^{{\bf r}}(T)= \sup\left\lbrace h_{\mu}^{{\bf r}}(T): \mu~ \text{is}~\Theta\text{-invariant} \right\rbrace,
\end{align*}
where $h_{\mu}^{{\bf r}}(T)$  and $h_{top}^{{\bf r}}(T)$ are the  measure-theoretic entropy and  topological entropy of random dynamical system, respectively. Ma, Yang and Chen \cite{ma2017} introduced the mean dimension and metric mean dimension for random   dynamical systems. Based on the work \cite{LWE,Gut,Shi,ma2017}, a natural question is  whether we can establish variational principle for metric mean dimension of random dynamical systems or not. So  we aim to  formulate variational principle for   metric mean dimension in the framework of continuous bundle random dynamical systems. On the  one hand, we address that  the role of measure-theoretic entropy is replaced by  the candidates of measure-theoretic $\epsilon$-entropies according to the definition of metric mean dimension. Borrowed the ideas of \cite{Gut, Shi}, we  give four types of measure-theoretic $\epsilon$-entropies, called  measure-theoretic entropy of partitions decreasing in diameter, Shapira's entropy, Katok's  $\epsilon$-entropy and Brin-Katok local $\epsilon$-entropy.  On the other hand, different from the techniques  used  for topological entropy, the local variational principles \cite{Bland, Huang, Gla, Ma, Zhang}  are the main  ingredient in our proof of  variational principles. Besides, more efforts are need   for the measurability of measure-theoretic $\epsilon$-entropies.

 This paper is organized as follows. In section \ref{sec1}, we recall the settings and  related notions of random dynamical systems  and the definition of  metric mean dimension for continuous bundle random dynamical systems. In section \ref{var},  we  prove the main results: Theorem \ref{main}, Theorem \ref{main1}, Theorem \ref{main2} and Theorem \ref{main3}.

\section{Preliminaries}\label{sec1}
\subsection{ The setup  of RDSs}
In this subsection, we recall the settings and  related notions of random dynamical systems  of continuous bundle random dynamical systems  investigated in \cite{AR,Cru,Kifer}.
 
Let  $(\Omega,\mathscr{F},\mathbb{P}, \theta)$ be a measure-preserving system, where $(\Omega, \mathscr{F}, \mathbb{P})$ is  countably generated probability space and $\theta$ is  invertible
measure-preserving transformation. We always assume that $\mathscr{F}$ is complete, countably generated, and separated points. Hence $(\Omega,\mathscr{F},\mathbb{P})$ is a Lebesgue space. Let $X$ be a compact metric space with the Borel $\sigma$-algebra  $\mathscr{B}_{X}$.  This  endows $\Omega \times X$ with the product $\sigma$-algebra $\mathscr{F}\otimes \mathscr{B}_X$.  For a measurable subset $\mc{E}\subset \Omega\times X$,   the fibers $\mc{E}_{\omega}=\left\lbrace x\in X: (\omega,x)\in \mc{E}\right\rbrace $ with $\omega\in \Omega$ are non-empty  compact subsets of $X$. A \emph{continuous bundle random dynamical system} (RDS for short) over $(\Omega, \mathcal{F}, \mathbb{P}, \theta)$ is generated by mappings $T_{\omega}: \mathcal{E}_{\omega}\rightarrow \mathcal{E}_{\theta \omega}$ with iterates
\begin{eqnarray}
T_{\omega}^{n}=\left\{\begin{array}{rcl}
&	T_{\theta^{n-1}\omega}\circ\cdots \circ T_{\theta \omega}\circ T_{\omega},&\text{if}~ n>0\\
	&id, &   \text{if}~ n=0
\end{array}\right.
\end{eqnarray}
so that  $(\omega, x)\mapsto T_{\omega}x$ is measurable and  $x\mapsto T_{\omega}x$ is continuous for $\mathbb{P}$-almost  all $\omega$. The map $\Theta:\mathcal{E}\rightarrow \mathcal{E}$ defined by $\Theta(\omega,x)=(\theta\omega, T_{\omega}x)$ is called the \emph{skew product transformation}.

By  $\mathcal{P}_{\mathbb{P}}(\Omega\times X)$  we denote  the space of probability measures on $\Omega\times X$ with the marginal $\mathbb{P}$ on $\Omega$. Let $\mathcal{P}_{\mathbb{P}}(\mathcal{E})=\left\lbrace \mu 
\in \mathcal{P}_{\mathbb{P}}(\Omega\times X): \mu(\mathcal{E})=1 \right\rbrace $.  It  is well-known that  \cite{Kifer} $\mu\in \mathcal{P}_{\mathbb{P}}(\mathcal{E})$ on $\mathcal{E}$ can be 
disintegrated  as $d \mu(\omega, x)=d \mu_{\omega}(x) d \mb{P}(\omega)$, where $\mu_{\omega}$ is the  regular conditional probabilities w.r.t.  the $\sigma$-algebra $\mathcal{F}_{\mathcal{E}}$ formed by all sets $(A\times X)\cap \mathcal{E}$ with $A\in \mathcal{F}$. The set of  $\Theta$-invariant measures $\mu\in \mathcal{P}_{\mathbb{P}}(\mathcal{E})$ is denoted by $M_{\mathbb{P}}(\mathcal{E})$. By Bogenschutz \cite{Bog}, the measure $\mu\in M_{\mathbb{P}}(\mathcal{E})$ if and only if $T_{\omega} \mu_{\omega}=\mu_{\theta \omega}$  for $\mb{P}$-a.e $\omega$. And the set of ergodic elements in $M_{\mathbb{P}}(\mathcal{E})$  is denoted by   $E_{\mathbb{P}}(\mathcal{E}) $.

\subsection{Metric mean dimension of RDSs}
In this subsection, we recall the definitions of  topological entropy \cite{Bog,Kifer} and metric mean dimension  introduced by Ma et al. \cite{ma2017} for continuous bundle random dynamical systems.

Let $\omega\in \Omega$, $n\in\mathbb{N}$ and $\epsilon>0$. For   each $x,y\in \mc{E}_{\omega}$,  the \emph{nth Bowen metric} $d_n^\omega$ on $\mc{E}_{\omega}$ is defined by 
\begin{equation*}
d_n^\omega(x,y)=\max\{ d( T_\omega^ix,T_\omega^iy ): 0\leq i<n \}.
\end{equation*}
Then the $(n,\epsilon, \omega)$-Bowen ball around $x$  with radius $\epsilon$ in the metric $d_{\omega}^{n}$  is given by 
$$B_{d_{n}^{\omega}}(x, \epsilon)=\left\lbrace y \in \mathcal{E}_{\omega}: d_{n}^{\omega}(x, y)<\epsilon\right\rbrace. $$ 

A set $E\subset\mc{E}_{\omega}$ is said to be  an \emph{$(\omega,\epsilon,n)$-separated set} if  $x,y\in E$, $x\neq y$ implies that $d_n^\omega(x,y)>\epsilon$. The maximum cardinality of $(\omega,\epsilon,n)$-separated sets is denoted by $\sep(\omega,\epsilon,n)$. A subset $F$ of $\mc{E}_{\omega}$ is said to be an \emph{$(\omega, \epsilon, n)$-spanning set} if for any $x\in \mc{E}_{\omega}$, there exists $y\in F$ such that $d_{n}^{\omega}(x, y)\leq \epsilon$. The smallest cardinality of  $(\omega, n, \epsilon)$-spanning sets is denoted by ${\rm span}(\omega, n,\epsilon)$.
Let 
$$S(\omega, \epsilon)=\limsup\limits_{n\rightarrow \infty} \frac{1}{n}\log {\rm sep}(\omega, \epsilon, n).$$
Notice that ${\rm sep}(\omega,\epsilon,n)$ is measurable in $\omega$ \cite[Lemma 2.1]{kfier2001}. The \emph{topological entropy for the  RDS} is defined by
\begin{align}\label{1}
	h_{top}^{\bf r}(T):=  \lim\limits_{\epsilon \rightarrow 0}\int S(\omega, \epsilon) d \mb{P}(\omega)=\sup_{\epsilon >0}\int S(\omega, \epsilon) d \mb{P}(\omega).
\end{align}
Set 
$$	h_{top}^{\bf r}(T,X,d,\epsilon)=\int \limsup\limits_{n\rightarrow \infty} \frac{1}{n}\log {\rm sep}(\omega,\epsilon, n) d \mathbb{P}(\omega).$$ 
Notice that $h_{top}^{\bf r}(T,X,d,\epsilon)$ is  non-decreasing as  $\epsilon\rightarrow 0$. One can   define  a quantity to  measure how fast the $\epsilon$-topological entropy  $h_{top}^{\bf r}(T,X,d,\epsilon)$ converges to  $	h_{top}^{\bf r}(T)$. For this purpose, we define the \emph{upper and lower metric mean dimensions of   $X$  for the RDS $T$} as follows
\begin{align*}
&\mathbb{E}\umdim(T,X,d)=\limsup\limits_{\epsilon \rightarrow 0}\frac{h_{top}^{\bf r}(T,X,d,\epsilon)}{|\log \epsilon|} ,\\
&\mathbb{E}\lmdim(T,X,d)= \liminf\limits_{\epsilon \rightarrow 0}\frac{h_{top}^{\bf r}(T,X,d,\epsilon)}{|\log \epsilon|} .
\end{align*}

Clearly,  the metric mean dimension  depends on the metrics  on $X$ and hence is not topological invariant. Furthermore,  one can deduce that any finite entropy systems have zero metric mean dimension. So  metric mean dimension is a useful quantity to describe the  topological complexity of infinite entropy systems.

\section{Variational principles for metric mean dimension}\label{var}

In this  section, we establish  four variational principles for metric mean dimension. The  main results are Theorem \ref{main}, Theorem \ref{main1}, Theorem \ref{main2} and Theorem \ref{main3}. 
\subsection{Variational principle I:  Kolmogorov-Sinai $\epsilon$-entropy}\label{sec2}
In this subsection, we first the local variational principle for the topological entropy  of a fixed finite open covers in terms of  measure-theoretic  entropy  of a fixed finite open covers given in \cite{Ma, Zhang}. Then we prove the first main result  Theorem \ref{main} by using the local variational principle.

A finite family  $\mathcal{U}=\left\lbrace U_{i}\right\rbrace_{i=1}^{k}$   of measurable subsets  of $\Omega \times X$ is  said to a  \emph{cover} if  $\Omega\times X=\bigcup_{i=1}^{k}U_{i}$,  and for each $i\in \left\lbrace 1,\cdots, k \right\rbrace $ the $\omega$-section $$U_{i}(\omega):=\left\lbrace x\in X: (\omega, x)\in U_{i} \right\rbrace $$ is a Borel set of $X$. This implies that $\mathcal{U}(\omega)=\left\lbrace U_{i}(\omega) \right\rbrace_{i=1}^{k}$ is a Borel cover of $X$. The sets of   \emph{partition} and \emph{open cover} of $\Omega\times X$, denoted by   $\mathcal{P}_{\Omega\times X}$ and $C_{\Omega\times X}^{0}$ respectively,
are the  cover of $\Omega\times X$ whose elements are pairwise disjoint, and the cover of $\Omega\times X$ whose elements are open sets.  Specially,   by $C_{\Omega\times X}^{0'}$ we denote the set of $\mc{U}\in C_{\Omega\times X}^{0}$ formed by $\mc{U}=\left\lbrace \Omega\times U_{i} \right\rbrace $ with   the finite  open cover $\left\lbrace U_{i}\right\rbrace $  of $X$. The notions $\mc{P}_{\mc{E}}$, $C_{\mc{E}}$, $C_{\mc{E}}^{0}$ and $C_{\mc{E}}^{0'}$   denote the restriction  of $\mc{P}_{\Omega\times X}$, $C_{\Omega\times X}$, $C_{\Omega\times X}^{0}$  and $C_{\Omega\times X}^{0'}$ on $\mc{E}$, respectively.   Given the covers  $\xi\in {C}_{\Omega}$ and $\mathcal{W}\in {C}_{X}$, we  sometimes write $(\Omega\times \mathcal{W})_{\mathcal{E}}=\{(\Omega \times W)\cap \mathcal{E}: W \in\mathcal{W} \}$ and $(\xi\times X)_{\mathcal{E}}=\{(A \times X)\cap \mathcal{E}: A\in\xi \}$.

Given two covers $\mathcal{U}$, $\mathcal{V}\in C_{\Omega\times X}$, $\mathcal{U}$ is said to be \emph{finer} than  $\mathcal{V}$ (write $\mathcal{U}\succeq \mathcal{V}$) if each element of $\mathcal{U}$ is contained in some element of $\mathcal{V}$. The join of  $\mathcal{U}$ and $\mathcal{V}$  is defined by  $\mathcal{U}\vee\mathcal{V}=\left\lbrace U\cap V: U\in \mathcal{U}, V\in \mathcal{V}\right\rbrace $. For $a, b\in \mathbb{N}$ with  $a\leq b$ and $\mathcal{U}\in C_{\Omega\times X}$, we set 
$\mathcal{U}_{a}^{b}=\bigvee\limits_{n=a}^{b}\Theta^{-n}\mathcal{U}.$ Let $x\in X$ and $\xi \in \mathcal{P}_\Omega$.
 Let $\xi(\omega)=\{A\cap \mathcal{E}_\omega:A\in \xi\}$ be the $\omega$-section of $\xi$ and$A_{\xi, \omega}^{n}(x)$  be  the member of the partition $\bigvee_{i=0}^{k-1} (T_{\omega}^{i})^{-1}\xi(\theta^{i}{\omega})$ to which $x$ belongs.

 Let $\mathcal{R}=\left\lbrace {R}_{i}\right\rbrace $ be a finite measurable  partition of $\mathcal{E}$ and ${R}_{i}(\omega)=\left\lbrace x \in \mathcal{E}_{\omega}: (\omega, x)\in {R}_{i} \right\rbrace $. Then   $\mathcal{R}(\omega)=\left\lbrace {R}_{i}(\omega) \right\rbrace $  is a  finite partition of $\mathcal{E}_{\omega}$. Set $\mathcal{F}_{\mathcal{E}}=\left\lbrace(A\times X)\cap \mathcal{E}: A\in \mathscr{F} \right\rbrace $.
The \emph{condition entropy} of $\mathcal{R}$  for  the given  $\sigma$-algebra $\mathcal{F}_{\mathcal{E}}$ is defined by
\begin{align*}
H_{\mu}(\mathcal{R}|\mathcal{F}_{\mathcal{E}})=-\int \sum_{i} \mu(R_{i}|\mathcal{F}_{\mathcal{E}})\log \mu (R_{i}|\mathcal{F}_{\mathcal{E}}) d \mathbb{P}(\omega)=\int H_{\mu_{\omega}}(\mathcal{R}(\omega)) d \mathbb{P}(\omega),	
\end{align*}
where $H_{\mu_{\omega}}(P)$ denotes the usual  partition entropy of $P$.  
Let  $\mu\in M_{\mathbb{P}}(\mathcal{E})$,  $\xi \in \mc{P}_{\mc{E}}$ and define 
\begin{align*}
h_{\mu}^{\bf r}({T}, \xi)&=\lim\limits_{n\rightarrow \infty}\frac{1}{n}H_{\mu}\left( \bigvee_{i=0}^{n-1}(\Theta^{i})^{-1}\xi|\mathcal{F}_{\mathcal{E}}\right) \\&=\lim\limits_{n\rightarrow \infty} \dfrac{1}{n} \int H_{\mu_{\omega}}\left( \bigvee_{i=0}^{n-1}(T_{\omega}^{i})^{-1} \xi(\theta^{i}\omega)\right)  d \mb{P}(\omega),
\end{align*} 
where the  limit exists dues to the  subadditivity of conditional entropy\cite{Kifer}.  If  $\mathbb{P}$ is ergodic,
then  $	h_{\mu}^{\bf r}(T,\xi)=\lim\limits_{n\rightarrow \infty} \dfrac{1}{n} H_{\mu_{\omega}}\left( \bigvee_{i=0}^{n-1}(T_{\omega}^{i})^{-1} \xi(\theta^{i}\omega)\right) $  for $\mathbb{P}$-a.e $\omega$.  

Let $\mathcal{U} \in {C}_{\mathcal{E}}^{0}$ and $\mu\in {M}_{\mathbb P}(T)$. We define the \emph{measure-theoretic entropy of  open cover $\mathcal{U}$ w.r.t. $\mu$}  as 
\begin{align*}
h_{\mu}^{\bf r}(T, \mathcal{U})=\inf_{\alpha\succeq \mathcal{U}, \alpha \in \mathcal{P}_{\mathcal{E}}} h_{\mu}^{\bf r}(T, \alpha).
\end{align*}
For each $\mathcal{U}\in {C}_{\mathcal{E}}^{0'}$, it is  not difficult to verify (see \cite{Kifer, Bog, Zhang}) that infimum above can only  take over  the partitions $Q$ of $\mathcal{E}$ into  sets  $Q_{i}$ of  the form $Q_{i}=(\Omega\times P_{i})\cap \mathcal{E}$, where $\mathcal{P}=\left\lbrace P_{i}\right\rbrace $ is a finite  partition of $X$.

Let $\mathcal{U}\in C_{\mathcal{E}}^{0}$, $n\in \mb{N}$ and  $\omega \in \Omega$. Put
\begin{align*}
N(T, \omega, \mathcal{U}, n)=\min \left\lbrace \#F: F ~\text{is the finite subcover of }~\bigvee_{i=0}^{n-1}(T_{\omega}^{i})^{-1}\mathcal{U}(\theta^{i}\omega)~\text{over}~\mathcal{E}_{\omega}\right\rbrace, 
\end{align*}
By the proof of \cite[Proposition 1.6]{kfier2001}, the quantity $N(T,\omega,\mathcal{U}, n)$ is measurable in $\omega$.
The Kingman's subadditive ergodic theorem gives us the following: 
\begin{align}\label{c}
h_{top}^{\bf r}(T, \mathcal{U}):&=\int\lim\limits_{n\rightarrow \infty}\frac{1}{n} \log N(T, \omega, \mathcal{U}, n) d \mathbb{P}(\omega)\\&\nonumber=\lim\limits_{n\rightarrow \infty} \frac{1}{n}\int \log N(T, \omega, \mathcal{U}, n)d \mathbb{P}(\omega),
\end{align}
and  (\ref{c}) remains true  for  $\mathbb{P}$-a.e $\omega$ without taking the integral in the right-hand side if $\mathbb{P}$ is ergodic. 

The authors \cite{Ma, Zhang}  established the following  local variational principle.
\begin{thm}\label{local}
Let $T$ be a continuous bundle RDS over a measure-preserving  system $(\Omega, \mc{F},\mb{P}, \theta)$. If  $\mathcal{U}\in C_{\mathcal{E}}^{0'}$, then 
\begin{align*}
h_{top}^{\bf r}(T, \mathcal{U})=\max\limits_{\mu\in M_{\mathbb{P}}(T)}h_{\mu}^{\bf r}(T, \mathcal{U}).
\end{align*}
Additionally, if $ \mathbb{P}$ is ergodic, then
\begin{align*}
h_{top}^{\bf r}(T, \mathcal{U})=\sup_{\mu \in E_{\mathbb{P}}(T)}h_{\mu}^{\bf r}(T, \mathcal{U}).
\end{align*}
\end{thm}

Given a finite open cover  $\mc{U}$ of $X$,  by  ${\rm diam}(\mc{U})$  we denote the  \emph{diameter} of  $\mc{U}$ , that is, the maximal diameter  of the elements of  $\mc{U}$. The Lebesgue number of  $\mc{U}$, denoted by ${\rm Leb}(\mc{U})$,  is  the largest  positive number $\delta$ with the property that every open ball of $X$ with radius  $\delta$ is contained in an element of $\mc{U}$. 

\begin{lem} \label{key}  \rm{\cite[Lemma 3.4]{Gut}}
Let $(X,d)$ be a compact metric space. Then for every $\epsilon>0$, there exists a finite open cover $\mathcal{U}$ of $X$ such that ${\rm diam}(\mc{U})\leq\epsilon$ and $\rm{Leb}(\mc{U})\geq \frac{\epsilon}{4}$.
\end{lem}
\begin{proof}
It follows by  considering  a $\frac{\epsilon}{4}$-net.
\end{proof}
\begin{lem}\label{ss}
Let   $\sigma=\left\lbrace A_{i}\right\rbrace $  be  a finite open cover  of $X$. Let $\mathcal{U}=(\Omega\times \sigma)_{\mathcal{E}}=\left\lbrace (\Omega\times A_{i}) \cap \mathcal{E}: A_i\in \sigma \right\rbrace $ be a finite open cover of $\mathcal{E}$. Then for  each fixed $\omega$,
\begin{align}\label{2}
S(\omega,{\rm diam}(\sigma) ) \leq \lim\limits_{n\rightarrow \infty}\frac{1}{n}\log N(T, \omega, \mathcal{U},n) \leq  S(\omega, \rm{Leb}(\sigma)).
\end{align}
\end{lem}
\begin{proof}
One can get the desired result by  using
 $${\rm sep}(\omega, {\rm diam}(\sigma), n)\leq N(T, \omega,\mathcal{U}, n) \leq {\rm sep} (\omega, {\rm Leb}(\mathcal{\sigma}), n).$$ 
\end{proof}

\begin{thm}\label{main}
Let $T$ be a continuous bundle RDS over a measure-preserving  system $(\Omega, \mc{F},\mb{P}, \theta)$. Then
\begin{align*}
&\mb{E}\umdim(T,X,d)=\limsup\limits_{\epsilon\rightarrow 0} \frac{1}{|\log \epsilon|} \sup_{\mu \in {M}_{\mathbb{P}}(T)} \inf_{\substack{{\rm diam}(\alpha)\leq \epsilon,\\ \alpha\in\mc{P}_{X}}} h_{\mu}^{\bf r}(T, (\Omega\times\alpha)_{\mathcal{E}}),\\&
\mb{E}\lmdim(T,X,d)=\liminf\limits_{\epsilon\rightarrow 0} \frac{1}{|\log \epsilon|} \sup_{\mu \in {M}_{\mathbb{P}}(T)} \inf_{\substack{{\rm diam}(\alpha)\leq {\epsilon},\\ \alpha\in\mc{P}_{X}}} h_{\mu}^{\bf r}(T, (\Omega\times \alpha)_{\mathcal{E}}).
\end{align*}
Additionally, if $(\Omega, \mc{F}, \mb{P}, \theta)$ is ergodic, then   the results   are also valid by changing the supremums  into $\sup_{\mu \in {E}_{\mathbb{P}}(T)}$.
\end{thm}
\begin{proof}
It suffices to show the variational principles hold  for  $\mb{E}\umdim(T,d)$. 
Let  $\epsilon>0$. From Lemma {\ref{key}}, there exists  a finite open cover $\mathcal{U}$ of $X$ such that ${\rm diam}(\mc{U})\leq\epsilon$ and ${\rm Leb}(\mc{U})\geq \frac{\epsilon}{4}$. 

Note that ${\rm diam}(\alpha)\leq\epsilon$ for any  finite Borel partition $\alpha$ of $X$ with  $\alpha\succeq \mathcal{U}$. By Theorem \ref{local}, we obtain 
\begin{align}\label{2.1}
\sup_{\mu \in {M}_{\mathbb{P}}(T)} \inf_{\substack{{\rm diam}(\alpha)\leq\epsilon, \\ \alpha\in \mathcal{P}_{X}}} h_{\mu}^{\bf r}(T, (\Omega\times\alpha)_{\mathcal{E}})\leq \sup_{\mu \in {M}_{\mathbb{P}}(T)} \inf_{\substack{\alpha\succeq \mc{U},\\a\in \mathcal{P}_{X}}} h_{\mu}^{\bf r}(T, (\Omega\times\alpha)_{\mathcal{E}})=h_{top}^{\bf r}(T, (\Omega\times\mc{U})_{\mathcal{E}}).
\end{align} 
Using Lemma \ref{ss}, 
\begin{align}\label{2.2}
h_{top}^{\bf r}(T,(\Omega\times \mc{U})_{\mathcal{E}}) \leq \int S(\omega, {\rm Leb}(\mathcal{U})) d \mathbb{P}(\omega)\leq \int S(\omega, \frac{\epsilon}{4}) d\mathbb{P}(\omega).
\end{align}
It follows from  inequalities  (\ref{2.1}) and (\ref{2.2}) that
\begin{align*}
\sup_{\mu \in {M}_{\mathbb{P}}(T)} \inf_{\substack{{\rm diam}(\alpha)\leq\epsilon, \\ \alpha\in \mathcal{P}_{X}}} h_{\mu}^{\bf r}(T, (\Omega\times\alpha)_{\mathcal{E}})\leq  \int S(\omega, \frac{\epsilon}{4}) d\mathbb{P}(\omega).
\end{align*}
So we get 
$$ \limsup\limits_{\epsilon\rightarrow 0} \frac{1}{|\log \epsilon|} \sup_{\mu \in {M}_{\mathbb{P}}(T)} \inf_{\substack{{\rm diam}(\alpha)\leq \epsilon, \\ \alpha\in\mc{P}_{X}}} h_{\mu}^{\bf r}(T, (\Omega\times\alpha)_{\mathcal{E}}) \leq \mb{E}\umdim(T,X,d).$$

On the other hand, for every   finite Borel partition $\alpha$  of $X$ such that ${\rm diam}(\alpha)\leq \frac{\epsilon}{8}$, one has $\alpha \succ \mathcal{U}$. Then the  Theorem \ref{local} and Lemma \ref{ss} give us
\begin{align*}
\sup_{\mu\in {M}_{\mathbb{P}}(T)}\inf_{\substack{{\rm diam}(\alpha)\leq \frac{\epsilon}{8},\\ \alpha\in\mc{P}_{X}}}h_{\mu}^{\bf r}(T,(\Omega\times\alpha)_{\mathcal{E}})&\geq \sup_{\mu \in {M}_{\mathbb{P}}(T)} \inf_{\substack{\alpha\succeq \mc{U},\\ a\in \mathcal{P}_{X}}}h_{\mu}^{\bf r}(T, (\Omega\times\alpha)_{\mathcal{E}})\\
&=h_{top}^{{\bf r}}(T, (\Omega\times\mathcal{U})_{\mc{E}})\\& \geq \int S(\omega, {\rm diam} (\mathcal{U})) d \mathbb{P}(\omega)\geq \int S(\omega,  \epsilon) d \mathbb{P}(\omega),
\end{align*} 
which  yields the desired results.  If $(\Omega, \mc{F}, \mb{P}, \theta)$ is ergodic, one can get the variational principles by the similar arguments.
\end{proof}
\subsection{Variational principle II: Shapira's  $\epsilon$-entropy}\label{sec3}
Let $\mathcal{U}=\left\lbrace U_{i} \right\rbrace_{i=1}^{k} $ be a finite open cover of $ \mathcal{E}$ and $\mu\in E_{\mathbb{P}}(T)$. Given $\omega\in \Omega$ and $0<\delta<1$, we define
$$N_{\mu_{\omega}}(\mathcal{U}, \delta)=\min \left\lbrace \#I: \mu_{\omega}\left( \bigcup_{i\in I} U_{i}(\omega)\right) > 1-\delta \right\rbrace.$$

\begin{prop}\label{11}
Let $T$ be a continuous bundle RDS over   a measure-preserving system $(\Omega, \mc{F},\mb{P}, \theta)$. Let  $\mathcal{U}\in C_{\mc{E}}^{0}$.	Then the  function $\omega \mapsto N_{\mu_{\omega}}(\mc{U}, \delta)$ is measurable.
\end{prop}
\begin{proof}
For every $q>0$, we have
	\begin{align*}
		&\Omega_q:=\left\lbrace \omega: N_{\mu_{\omega}}(\mc{U}, \delta)=q\right\rbrace \\
		=&\bigcup_{\#I=q,\atop
			I \subset \left\lbrace 1,\cdots, \#\mc{U}
			\right\rbrace }\left\lbrace\omega:\mu_{\omega}\left(\bigcup_{i\in I} {U}_{i}(\omega) \right)>1-\delta  \right\rbrace \bigcap \left( \bigcap_{\#J<q, \atop
			J\subset \left\lbrace 1,\cdots, \#\mc{U}\right\rbrace}\left\lbrace \omega: \mu_{\omega}\left(\bigcup_{i\in J} {U}_{i}(\omega) \right)\leq 1-\delta \right\rbrace \right).
	\end{align*}
For each $I\subset  \{1,\cdots, \#\mc{U}\}$, the  graph($A_I$)$=\{(\omega,x):x\in \bigcup_{i\in I} {U}_{i}(\omega) \}=\cup_{i\in I}U_i \cap \mathcal{E}$ is a measurable set of $\Omega \times X$. By \cite[Corollary 3.4]{Cru}, the map $\omega \rightarrow \mu_{\omega}(U_{i}(\omega))$ is measurable. Then  $\Omega_q$ is a measurable set of $\Omega$. This implies that   $\omega\mapsto N_{\mu_{\omega}}(\mc{U}, \delta)$ is measurable since the map only takes finite many values. 
\end{proof} 
Using Proposition \ref{11}, we  can  define \emph{Shapira's entropy of $\mathcal{U}\in C_{\mc{E}}^{0}$ w.r.t. $\mu$} as
\begin{align*}
	&\overline{h}_{\mu}^{S}(T, \mathcal{U}):=  \lim\limits_{\delta \rightarrow 0} \limsup\limits_{n\rightarrow \infty}\dfrac{1}{n}\int  \log  N_{\mu_{\omega}}(\mathcal{U}_{0}^{n-1}, \delta) d \mathbb{P}(\omega).\\&
	\underline{h}_{\mu}^{S}(T, \mathcal{U}):= \lim\limits_{\delta \rightarrow 0}  \liminf\limits_{n\rightarrow \infty}\dfrac{1}{n}\int \log  N_{\mu_{\omega}}(\mathcal{U}_{0}^{n-1}, \delta) d \mathbb{P}(\omega).
\end{align*}

Adapting from the ideas  from \cite{Sha} \cite{Wu}, the following Theorem  establishes  the bridge between Shapira's entropy and measure-theoretic entropy of a fixed finite open cover $\mathcal{U}$ for  random dynamical systems.
\begin{thm}\label{local variational principle}
	Let $T$ be a continuous bundle RDS over a measure-preserving  system $(\Omega, \mc{F},\mb{P}, \theta)$.	Let $\mathcal{U}\in C_\mathcal{E}^{0}$ and $\mu\in E_{\mathbb{P}}(T)$. Then 
	\begin{align*}
		\overline{h}_{\mu}^{S}(T, \mathcal{U})=\underline{h}_{\mu}^{S}(T, \mathcal{U})=h_{\mu}^{\bf r}(T, \mathcal{U}).
	\end{align*} 
\end{thm}

\begin{proof}
	{\bf Step 1}: We prove $h_{\mu}^{\bf r}(T, \mathcal{U})\geq \overline{h}_{\mu}^{S}(T, \mathcal{U})$. 
	
	Take any a finite measurable partition $\xi$ of $\mathcal{E}$ such that $\xi \succeq \mathcal{U}$.	
	According to  Lemma \ref{SMB}, there exists $F\subset \mathcal{E}$ such that $\mu(F)=1$ and for each $(\omega, x)\in F$, 
	\begin{align*}
		\lim\limits_{n\rightarrow \infty} -\dfrac{1}{n} \log \mu_{\omega}(A_{\xi, \omega}^{n}(x))=h_{\mu}^{\bf r}(T, \xi).
	\end{align*}
	Fix $\omega \in \pi_{\Omega}(F)$ and let $a>0$.   Set
	\begin{align*}
		L_{\omega, n}=\left\lbrace x\in \mathcal{E}_{\omega}: -\dfrac{1}{m}\log \mu_{\omega}(A_{\xi,\omega}^{m}(x))\leq h_{\mu}^{\bf r}(T, \xi)+a, \forall m\geq n \right\rbrace.
	\end{align*}
	By Lemma \ref{SMB},  $\mu_{\omega}(L_{\omega, n})> 1-\delta$ for $n$ sufficiently large.  Fix $n$ and  choose  a finite subset $G_{\omega, n}=\left\lbrace x_{1}, \cdots, x_{s_{\omega, n}} \right\rbrace $ of $L_{\omega, n}$ such that $L_{\omega, n}\subset \bigcup_{i=1}^{s_{\omega, n}}A_{\xi, \omega}^{n}(x_{i})$. Since the sets $A_{\xi, \omega}^{n}(x_{i})$ are distinct  and $\mu_{\omega}$ measure of each member of them is not less than $\exp(-n(h_{\mu}^{\bf r}(T, \xi)+a))$, then
	\begin{align*}
		\#G_{\omega, n}=s_{\omega, n}\leq \exp(n(h_{\mu}^{\bf r}(T, \xi)+a)).
	\end{align*}  
	Note that $\mu_{\omega} (L_{\omega, n})> 1-\delta$, we have 
	\begin{align}
		N_{\mu_{\omega}}(\mathcal{U}_{0}^{n-1}, \delta)\leq N_{\mu_{\omega}}(\xi^{n}, \delta)\leq \exp(n(h_{\mu}^{\bf r}(T, \xi)+a)).
	\end{align}
	Thus for any $a>0$
	\begin{align*}
 \limsup\limits_{n \rightarrow \infty}\frac{1}{n}\int  \log N_{\mu_{\omega}}(\mc{U}_{0}^{n-1}, \delta) d \mb{P}(\omega) \leq	h_{\mu}^{\bf r}(T, \xi)+a.
	\end{align*}
Letting $a\rightarrow 0$, we obtain
	 $$\limsup\limits_{n \rightarrow \infty}\frac{1}{n}\int  \log N_{\mu_{\omega}}(\mc{U}_{0}^{n-1}, \delta) d \mb{P}(\omega)\leq h_{\mu}^{\bf r}(T, \xi).$$ Taking infimum over $\xi\succeq \mathcal{U}$ and $\delta\rightarrow 0$, we have 
	\begin{align*}\label{key1}
	\overline{h}_{\mu}^{S}(T, \mathcal{U}) \leq 	h_{\mu}^{\bf r}(T, \mathcal{U}) .
	\end{align*}
\end{proof}
\begin{claim}\label{lem}
	For any $\mathcal{V} \in {C}_{\mathcal{E}}^{0}$ and $0<\delta<1$, there exists  $\beta \in \mathcal{P}_{\mathcal{E}}$ such that $\beta \succeq \mathcal{V}$ and $N_{\mu_{\omega}}(\beta, \delta)\leq N_{\mu_{\omega}}(\mathcal{V}, \delta)$ for $\mb{P}$-a.e. $\omega\in \Omega$.
\end{claim}
\begin{proof}
	Let $\mathcal{V}=\left\lbrace V_{1}, \cdots, V_{m} \right\rbrace \in C_{\mc{E}}^{0}$. For $\mathbb{P}$-a.e. $\omega \in \Omega$, there exists $I_{\omega}\subset\left\lbrace 1, \cdots, m \right\rbrace $ with cardinality $N_{\mu_{\omega}}(\mathcal{V}, \delta)$ such that $\mu_{\omega}(\bigcup_{i\in I_{\omega}} V_{j}(\omega))\geq 1-\delta.$  Hence we can find $w_{1}, \cdots, w_{s}\in  \Omega$ such that for $\mathbb{P}$-a.e.$\omega\in \Omega$, $I_{\omega}=I_{\omega_{i}}$ for some $i\in \left\lbrace 1, \cdots, s\right\rbrace $. For $i=1, \cdots, s$, define
	\begin{align*}
		\Omega_{i}=\left\lbrace \omega\in \Omega: \mu_{\omega}(\bigcup_{j\in I_{w_{i}}}V_{j}(\omega)) \geq 1-\delta\right\rbrace. 
	\end{align*}
	Let $C_{1}=\Omega_{1}$, $C_{i}=\Omega_{i}\backslash \bigcup_{j=1}^{i-1} \Omega_{j}$, $i=2, \cdots, s$.
	Fix $i\in \left\lbrace 1,\cdots, s\right\rbrace $. Assume that $I_{\omega_{i}}=\left\lbrace k_{1}, \cdots, k_{t_{i}} \right\rbrace $, where $t_{i}=N_{\mu_{\omega_{i}}}(\mathcal{V}, \delta)$.  Take $\left\lbrace W_{1}^{\omega_{i}}, \cdots, W_{t_{i}}^{\omega_{i}} \right\rbrace $ such that 
	\begin{align*}
		W_{1}^{\omega_{i}}=V_{k_{1}}, W_{2}^{\omega_{i}}=V_{k_{2}}\setminus V_{k_{1}}, \cdots, W_{t_{i}}^{\omega_{i}}=V_{k_{t_{i}}}\setminus \cup_{j=1}^{t_{i}-1}V_{k_{j}}.
	\end{align*}
	Define $A:=\mathcal{E}\setminus \left( \cup_{i=1}^{s}(\pi_{\Omega}^{-1}C_{i}\cap \cup_{j=1}^{t_{i}} W_{j}^{\omega_{j}}) \right) $. Set $A_{1}=A\cap V_{1}$, $A_{l}:=A\cap (V_{l}\setminus \cup_{j=1}^{l-1}V_{j})$, $l=2, \cdots, m$. Finally, take
	\begin{align*}
		\beta=\left\lbrace \pi_{\Omega}^{-1}C_{1}\cap W_{1}^{\omega_{1}}, \cdots, \pi_{\Omega}^{-1}C_{1}\cap W_{t_{1}}^{\omega_{1}}, \cdots, 
		\pi_{\Omega}^{-1}C_{s}\cap W_{1}^{\omega_{s}}, \cdots, \pi_{\Omega}^{-1}C_{s}\cap W_{t_{s}}^{\omega_{s}}, A_{1}, \cdots, A_{m}\right\rbrace.
	\end{align*}
	Then $\beta \succeq \mathcal{V}$ and $N_{\mu_{\omega}}(\beta, \delta)\leq N_{\mu_{\omega}}(\mathcal{V}, \delta)$ for $\mb{P}$-a.e. $\omega$.
\end{proof}
The following lemma is the strong  Rohlin Lemma.
\begin{lem}\label{cc}{\rm \cite[Lemma 2.2]{Sha}}
	Let $(X, \mathcal{B}, \mu, T)$ be an ergodic, aperiodic system and let $\alpha\in \mathcal{P}_{X}$. Then for any $\delta>0$ and $n\in \mathbb{N}$, one can find a set $B\in \mathcal{B}$ such that $B$, $TB$, $\cdots$, $T^{n-1}B$ are mutually disjoint, $\mu\left( \bigcup_{i=0}^{n-1} T^{i} B\right)>1-\delta $ and the distribution of $\alpha$ is the same as the distribution of the partition $\alpha|_{B}$ that $\alpha$ induces on $B$.
\end{lem}
\noindent{\bf Step 2:} We prove 
$$h_{\mu}^{\bf r}(T, \mathcal{U})\leq \underline{h}_{\mu}^{S}(T, \mathcal{U})$$ for ergodic measure $\mu$ and $\mathcal{U}\in C_\mathcal{E}^{0}$.
\begin{proof}
	Fix $n\in \mathbb{N}$. By Claim \ref{lem}, we can find $\beta \in \mathcal{P}_{\mathcal{E}}$ such that $\beta \succeq \mathcal{U}_{0}^{n-1}$ and there exists a subset $A$ of $\mathcal{E}$   such that  $\mu(A)<\rho$ and for any $(\omega,x)\notin A$, we have $N_{\mu_{\omega}}(\beta, \rho)\leq N_{\mu_{\omega}}(\mathcal{U}_{0}^{n-1}, \rho)$. Pick $\delta>0$ such that $0<\rho +\delta<1/4$. By  Lemma \ref{cc}, we can construct a strong Rohlin tower with respect to $\beta$, with height  $n$ and error $<\delta$. Let $\tilde{B}$ denote the  base of  tower and $B=\tilde{B}\setminus A$. Clearly, $\mu(B)>(1-\rho)\mu(\tilde{B})$ and $\mu(E)\geq 1- (\rho+\delta)$, where $E=\cup_{i=0}^{n-1}\Theta^{i}B$. Consider $\beta|_{\tilde{B}}$ and index its elements by sequences $i_{0}, \cdots, i_{n-1}$ such that if $B_{i_{0}, \cdots, i_{n-1}}\in \beta|_{\tilde{B}},$ then $\Theta^{j} B_{i_{0}, \cdots, i_{n-1}}\subset U_{i_{j}}$ for every $0\leq j \leq n-1$. Let $\hat{\alpha}=\left\lbrace \hat{A}_{1}, \cdots, \hat{A}_{M}\right\rbrace $ be a partition of $E$ defined by 
	\begin{align*}
		\hat{A}_{m}:=\bigcup\limits \left\lbrace \Theta^{j} B_{i_{0}, \cdots, i_{n-1}}: 0\leq j\leq n-1, i_{j}=m \right\rbrace. 
	\end{align*}
	Note that  $\hat{A}_{m}\subset U_{m}$ for every  $1\leq m \leq M$. Extend $\hat{\alpha}$ to a partition $\alpha$ of $\mathcal{E}$ in some way such that $\alpha\succeq \mathcal{U}$ and $\#\alpha=2M$. 
	Set $\eta^{4}=\rho+\delta$ and define for every $k>n$ large enough, $f_{k}(\omega, x)=\frac{1}{k}\sum_{i=0}^{k-1}1_{E}(\Theta^{i}(\omega, x))$ and $L_{k}:=\left\lbrace (\omega, x)\in \mathcal{E}: f_{k}(\omega, x)>1-\eta^{2} \right\rbrace $. Then by Birkhoff ergodic theorem $\int f_{k} d \mu >1-\eta^{4}$ and 
	\begin{align*}
		\eta^{2}\mu(L_{k}^{c})\leq \int_{L_{k}^{c}} 1-f_{k} d \mu \leq \int_{\mathcal{E}} 1-f_{k} d \mu \leq \eta^{4}.
	\end{align*}
	Then $\mu(L_{k})\geq  1-\eta^{2}$.
	Take
	\begin{align*}
		&J_{k}=\left\lbrace(\omega, x)\in \mathcal{E}: \mu_{\omega}(A_{\alpha,\omega}^{j}(x))<\exp(-(h_{\mu}^{\bf r}(T, \alpha)-\eta)j), \forall j\geq k \right\rbrace \bigcap \\&\left\lbrace (\omega,x)\in \mc{E}: \bigg|\frac{1}{j}\sum_{i=0}^{j-1}\log N_{\mu_{\theta^{i}\omega}}(\mathcal{U}_{0}^{n-1}, \rho)1_{B}(\Theta^{i}(\omega, x))-\int_{B}
		\log N_{\mu_{\omega}}(\mathcal{U}_{0}^{n-1}, \rho) d \mu\bigg |\leq \eta, \forall j\geq k\right\rbrace. 
	\end{align*} 
	By Theorem  \ref{SMB} and Birkhoff ergodic theorem, $\mu(J_{k})>1-\eta^{2}$  for $k$ large enough. Set $G_{k}=L_{k}\cap J_{k}$ and then  $\mu(G_{k})>1-2\eta^{2}$. Define
	\begin{align*}
		\tilde{G}_{k}^{c}=\left\lbrace(\omega, x)\in G_{k}: \mu_{\omega}(G_{k})<1-4\eta \right\rbrace\cup G_{k}^{c}=\left\lbrace (\omega, x)\in G_{k}: \mu_{\omega}(G_{k}^{c})>4\eta \right\rbrace \cup G_{k}^{c}. 
	\end{align*}
	Therefore,
	\begin{equation*}
		\mu(\tilde{G}_{k}^{c})\cdot 4 \eta \leq  \int \mu_{\omega} (G_{k}^{c}) d \mu+\mu(G_{k}^{c})=2  \mu(G_{k}^{c})\leq 4 \eta^{2},
	\end{equation*}
i.e., $	\mu(\tilde{G}_{k}^{c})\leq \eta$.
	Given $(\omega, x)\in \mathcal{E}$, we fix an element $C$ of this partition of $G_{k}\cap \pi_{\Omega}^{-1} \pi_{\Omega} (\omega, x)$ and want to estimate the number of $\alpha_{0}^{n-1}$-elements of $C$ visit $B$, then we need at most $N_{\mu_{\theta^{i_{j}}\omega}}(\mathcal{U}_{0}^{n-1}, \rho)$ $\alpha_{i_{j}}^{i_{j}+n-1}$-elements to cover $C$ for each $(\omega ,x)\in G_{k}$. Because the size of $[0, k-1]\setminus \cup_{j}[i_{j}, i_{j}+n-1]$ is at most $\eta^{2}k+2n,$ we need at most $\Pi_{j=1}^{m} N_{\mu_{\theta^{j}\omega}}(\mathcal{U}_{0}^{n-1}, \rho)\cdot (2M)^{\eta^{2}k+2n}$ $\alpha^{k-1}$-elements to cover $C$. Since $(\omega, x) \in G_{k}$, we have
	$G_{k}\cap \pi_{\Omega}^{-1}\omega$ can be covered by no more than 
	\begin{align*}
		e^{kH(\eta^{2}+2n/k)}\cdot (2M)^{\eta^{2}+2n/k}\cdot e^{k(\int_{B} \log N_{\mu_{\omega}}(\mathcal{U}_{0}^{n-1}, \rho)d \mu+\eta)}
	\end{align*}
	$\alpha_{k-1}$ elements. Note that $(\omega, x)\in G_{k}$, then we have 
	\begin{align}\label{aa}
		&1-4\eta\leq \mu_{\omega}(G_{k}\cap \pi_{\Omega}^{-1}\omega)\leq \\& \exp(-(h_{\mu}^{\bf r}(T, \alpha)-\eta)k)\cdot\exp(kH(\eta^{2}+2n/k))\cdot \nonumber\\&(2M)^{\eta^{2}+2n/k}\cdot \exp\left( k(\int_{B} \log N_{\mu_{\omega}}(\mathcal{U}_{0}^{n-1}, \rho)d \mu+\eta)\right).\nonumber
	\end{align}
	By Claim \ref{lem}, we can get $(\omega, x)\rightarrow N_{\mu_{\omega}}(\mathcal{U}_{0}^{n-1}, \rho)$ is constant on each atom of $\beta|_{\mathcal{E}}$. Note that the distribution of $\beta$ is the same as the distribution of   partition $\beta|_{\tilde{B}}$. Combining with (\ref{aa}), we have 
	\begin{align*}
		h_{\mu}^{\bf r}(T, \alpha) &\leq \eta +H(\eta^{2})+\eta^{2}\log (2M)+ \int_{B} \log N_{\mu_{\omega}}(\mathcal{U}_{0}^{n-1}, \rho) d \mu+\eta \\&\leq 2\eta+ H(\eta^{2})+\eta^{2}\log (2M)+\frac{1}{n}\int_{B}\log N_{\mu_{\omega}}(\mathcal{U}_{0}^{n-1}, \rho) d \mu\\& \leq 2\eta+ H(\eta^{2})+\eta^{2}\log (2M)+\frac{1}{n}\int_{\pi_{\Omega}(B)}\log N_{\mu_{\omega}}(\mathcal{U}_{0}^{n-1}, \rho) d \mathbb{P}(\omega).
	\end{align*} 
	Letting $n\rightarrow \infty$ and then $\rho \rightarrow 0$, we have
	\begin{align*}
		h_{\mu}^{\bf r}(T, \mathcal{U})\leq h_{\mu}^{\bf r}(T, \alpha) \leq \underline{h}_{\mu}^{S}(T, \mathcal{U}).
	\end{align*}
\end{proof}

\begin{thm}\label{main1}
Let $T$ be a continuous bundle RDS  over an ergodic measure-preserving system  $(\Omega, \mc{F},\mb{P}, \theta)$. Then
\begin{align*}
&\mb{E}\umdim(T,X,d)=\limsup\limits_{\epsilon \rightarrow 0}\frac{1}{|\log \epsilon|}\sup_{\mu\in {E}_{\mathbb{P}}(T)}\inf_{\substack{{\rm diam}(\mc{U})\leq \epsilon ,\\ \mc{U}\in \mc{C}_{X}^o }}h_{\mu}^{S}(T, (\Omega\times\mc{U})_{\mathcal{E}}).\\&
\mb{E}\lmdim(T,X, d)=\liminf\limits_{\epsilon \rightarrow 0} \frac{1}{|\log \epsilon|}\sup_{\mu\in {E}_{\mathbb{P}}(T)}\inf_{\substack{{\rm diam}(\mc{U})\leq \epsilon,\\ \mc{U}\in \mc{C}_{X}^o }}h_{\mu}^{S}(T, (\Omega\times\mc{U})_{\mathcal{E}}).
\end{align*}
\end{thm}

\begin{proof}
Fix  $\epsilon>0$ and $\mu
\in E_{\mathbb{P}}(T)$. Then 
\begin{align}\label{inequ 3.8}
\inf_{\substack{{\rm diam}(\mc{U})\leq \epsilon,\\ \mc{U}\in {C}_{X}^o }} h_{\mu}^{S}(T, (\Omega\times\mc{U})_{\mathcal{E}})
&=\inf_{{\rm diam}(\mc{U})\leq \epsilon, \atop  \mc{U}\in {C}_{X}^o  }h_{\mu}^{\bf r}(T, (\Omega\times\mc{U})_{\mathcal{E}}), \text{by Theorem}~\ref{local variational principle} \nonumber\\
&=\inf_{{\rm diam}(\mc{U})\leq \epsilon, \alpha \succeq \mc{U}}h_{\mu}^{\bf r}(T, (\Omega\times\alpha)_{\mathcal{E}})\nonumber\\
&\geq  \inf_{{\rm diam}(\alpha)\leq \epsilon, \atop \alpha\in {P}_{X}}h_{\mu}^{\bf r}(T,(\Omega\times\alpha)_{\mathcal{E}}).
\end{align}

By Lemma \ref{key}, we can choose a finite open cover $\mathcal{U}'$ of $X$ with ${\rm diam}(\mathcal{U}')\leq \epsilon$ and ${\rm Leb}(\mathcal{U}')\geq \frac{\epsilon}{4}$. Then 
\begin{align}\label{inequ 3.9}
\inf_{\substack{{\rm diam}(\mc{U})\leq \epsilon\\ \mc{U}\in {C}_{X}^o }} h_{\mu}^{S}(T, (\Omega\times\mc{U})_{\mathcal{E}})
&\leq  h_{\mu}^{S}(T, (\Omega\times\mc{U'})_{\mathcal{E}})\nonumber\\
&=h_{\mu}^{\bf r}(T, (\Omega\times\mc{U'})_{\mathcal{E}})=\inf_{\alpha \succeq \mc{U}^{'},\mc{\alpha}\in {P}_{X}}h_{\mu}^{\bf r}(T, (\Omega\times\alpha)_{\mathcal{E}}), \text{by Theorem}~ \ref{local variational principle}\nonumber\\
&\leq  \inf_{{\rm diam}(\alpha)\leq \frac{\epsilon}{8}, \atop \mc{\alpha}\in {P}_{X}}h_{\mu}^{\bf r}(T,(\Omega\times\alpha)_{\mathcal{E}}).
\end{align}
We finally get the desired results by the inequalities (\ref{inequ 3.8}), (\ref{inequ 3.9}) and Theorem \ref{main}.		
\end{proof}
\subsection{Variational principle III:  Katok's $\epsilon$-entropy}\label{sec4}
In this subsection,   replacing  Shapira's $\epsilon$-entropy with Katok local  $\epsilon$-entropy  we prove the three main result  Theorem \ref{main2} by virtue of  Theorem \ref{main1}.

Given $\mu\in {M}_{\mathbb{P}}(T)$, let
\begin{align*}
N_{\mu_{\omega}}^{\delta}(n, \epsilon)=\min\left\lbrace \# j: \mu_{\omega}\left( \bigcup_{i=1}^{j}B_{d_{n}^{\omega}}(x_{i}, \epsilon)\right) > 1-\delta \right\rbrace. 
\end{align*}
\begin{prop} \label{prop 3.2}
Let $\mu\in {M}_{\mathbb{P}}(T)$, $\epsilon>0$ and $0<\delta <1$. Then for every $n\geq 1$, the map $\omega \mapsto N_{\mu_{\omega}}^{\delta}(n, \epsilon)$ is measurable.
\end{prop} 
\begin{proof}
Measurability of  $N_{\mu_{\omega}}^{\delta}(n, \epsilon)$ is derived from  Proposition \ref{11}.
\end{proof}
Based on the Proposition \ref{prop 3.2}, we  define the \emph{upper and lower Katok's $\epsilon$-entropies of $\mu$} as follows
\begin{align*}
	&\overline{h}_{\mu}^{K}(T, \epsilon)=\lim\limits_{\delta \to0}\limsup\limits_{n\rightarrow \infty}\frac{1}{n}\int \log N_{\mu_{\omega}}^{\delta}(n, \epsilon) d \mathbb{P}(\omega),\\
	&
	\underline{h}_{\mu}^{K}(T, \epsilon)=\lim\limits_{\delta \to0}\liminf\limits_{n\rightarrow \infty}\frac{1}{n}\int \log N_{\mu_{\omega}}^{\delta}(n, \epsilon) d \mathbb{P}(\omega).
\end{align*}

\begin{thm}\label{main2}
Let $T$ be a continuous bundle RDS  over an ergodic measure-preserving  system $(\Omega, \mc{F},\mb{P}, \theta)$. Then	
\begin{align*}
&\mb{E}\umdim(T,X,d)=\limsup\limits_{\epsilon \rightarrow 0}\frac{1}{|\log \epsilon|}\sup_{\mu\in {E}_{\mathbb{P}}(T)}\overline{h}_{\mu}^{K}(T, \epsilon).\\
&
\mb{E}\lmdim(T, X, d)=\liminf\limits_{\epsilon\rightarrow 0} \frac{1}{|\log \epsilon|}\sup_{\mu\in {E}_{\mathbb{P}}(T)}\overline{h}_{\mu}^{K}(T, \epsilon).
\end{align*}
The results are valid if we change $\overline{h}_{\mu}^{K}(T, \epsilon)$ into $\underline{h}_{\mu}^{K}(T, \epsilon)$.
\end{thm}

\begin{proof}
It suffices to show the results hold for  $\mb{E}\umdim(T,X,d)$ since the second  one follows similarly. 
Fix $\epsilon >0$. Let $0<\delta<1$ and $\mu\in {E}_{\mathbb{P}}(T)$. Let $\mathcal{U}=\left\lbrace U_{1},\cdots, U_{l} \right\rbrace $ be a finite open cover of ${X}$ with ${\rm diam}(\mathcal{U})<\epsilon$. Then the family $\mathcal{U}(\omega)$  formed by the sets $U\cap \mathcal{E}_{\omega}$ with $U\in \mathcal{U}$ is an   open cover of $\mc{E}_{\omega}$.  This implies that each  element of $\bigvee_{i=0}^{n-1}(T_{\omega}^{i})^{-1}\mathcal{U}(\theta^{i}\omega)$  can be contained in an $(n, \epsilon, \omega)$-Bowen ball. So 
\begin{align*}
N_{\mu_{\omega}}^{\delta}(n, \epsilon) \leq N_{\mu_{\omega}}\left( \bigvee_{i=0}^{n-1}(T_{\omega}^{i})^{-1}\mathcal{U}(\theta^{i}\omega), \delta\right). 
\end{align*}
This shows
\begin{align}\label{inequ 3.10}
\overline{h}_{\mu}^{K}(T, \epsilon)\leq \inf_{\substack{{\rm diam}(\mc{U})\leq \epsilon,\\ \mc{U}\in \mc{C}_{X}^o }}h_{\mu}^{S}(T, (\Omega\times\mathcal{U})_{\mc{E}}).
\end{align}

By Lemma \ref{key} again, we can choose  a finite cover $\mathcal{U}$ of $X$ such that ${\rm diam}(\mc{U})\leq\epsilon$ and ${\rm Leb}(\mc{U})\geq \frac{\epsilon}{4}$. Since each $(n, \frac{\epsilon}{4}, \omega)$-Bowen ball is contained in some element of $\bigvee_{i=0}^{n-1}(T_{\omega}^{i})^{-1}\mathcal{U}(\theta^{i}\omega)$, then $N_{\mu_{\omega}}(\bigvee_{i=0}^{n-1}(T_{\omega}^{i})^{-1}\mathcal{U}(\theta^{i}\omega), \delta)\leq N_{\mu_{\omega}}^{\delta}(n, \frac{\epsilon}{4})$. 
This shows
\begin{align}\label{inequ 3.11}
 \inf_{\substack{{\rm diam}(\mc{U})\leq \epsilon,\\ \mc{U}\in \mc{C}_{X}^o }}h_{\mu}^{S}(T, (\Omega\times\mathcal{U})_{\mc{E}}) \leq 	\overline{h}_{\mu}^{K}(T, \frac{\epsilon}{4}).
\end{align}
Therefore, by inequalities (\ref{inequ 3.10}), (\ref{inequ 3.11}) and  Theorem \ref{main1},  we get the desired results. 
\end{proof}

\subsection{Variational principle  IV: Brin-Katok local $\epsilon$-entropy}\label{sec5}
In this subsection, we  borrow Shannon-McMillan-Breiman theorem of random dynamical systems and  Theorem \ref{main2} to establish the fourth variational principle  for metric mean dimensions in terms of Brin-Katok local  $\epsilon$-entropy.   

Let $\mu \in M_{\mathbb{P}}(T)$, $x\in \mathcal{X}$ and $\omega\in \Omega$.  Put
\begin{align*}
&\overline{h}_{\mu_{\omega}}^{BK}(T,x, \epsilon)=\limsup\limits_{n\rightarrow \infty}-\frac{1}{n}\log \mu_{\omega}(B_{d_{n}^{\omega}}(x, \epsilon)),\\&
\underline{h}_{\mu_{\omega}}^{BK}(T,x, \epsilon)=\liminf\limits_{n\rightarrow \infty}-\frac{1}{n}\log \mu_{\omega}(B_{d_{n}^{\omega}}(x, \epsilon)).
\end{align*}



We define  the  \emph{upper and lower Brin-Katok local  $\epsilon$-entropies of $\mu$ at $x$} as
\begin{align*}
\overline{h}_{\mu}^{BK}(T, \epsilon)&=\int \overline{h}_{\mu_{\omega}}^{BK}(T, x , \epsilon) d \mu,\\
\underline{h}_{\mu}^{BK}(T, \epsilon)&=\int \underline{h}_{\mu_{\omega}}^{BK}(T, x , \epsilon) d \mu.
\end{align*}
The Brin-Katok's entropy formula for RDS is given by Zhu  in \cite[Theorem 2.1]{Zhu}.
\begin{prop}\label{prop 3.2} Let $T$ be a continuous bundle RDS over  a measure-preserving  system $(\Omega, \mc{F},\mb{P}, \theta)$. If $\mu\in {E}_{\mathbb{P}}(T)$, then for every $\epsilon>0$,
\begin{align}\label{eq5}
	\overline{h}_{\mu_{\omega}}^{BK}(T, x, \epsilon)=\overline{h}_{\mu}^{BK}(T, \epsilon) ~\text{and}~ \underline{h}_{\mu_{\omega}}^{BK}(T, x, \epsilon)=\underline{h}_{\mu}^{BK}(T,  \epsilon)
\end{align}
for $\mu$-a.e  $(\omega, x)$.
\end{prop}
\begin{proof}
Let $d \mu(\omega, x)=d \mu_{\omega}(x) d \mb{P}(\omega)$ be disintegration    of $\mu\in E_{\mathbb{P}}(\mathcal{E})$ on $\mathcal{E}$. Let $F(\omega,x):=	\overline{h}_{\mu_{\omega}}^{BK}(T, x, \epsilon)$. Fix $n$. Then 
\begin{align*}
B_{d_{n}^{\omega}}(x, \epsilon)&=\cap_{j=0}^{n-1}(T_\omega^j)^{-1}(B(T_\omega^jx,\epsilon)\cap \mathcal{E}_{\theta^j\omega})\\
&=T_\omega^{-1}\cap_{j=0}^{n-1}(T_{\theta\omega}^{j-1})^{-1}(B(T_{\theta \omega}^{j-1}(T_\omega x),\epsilon)\cap \mathcal{E}_{\theta^{j-1}{\theta \omega}}), \text{by}~T_\omega^jx=T_{\theta \omega}^{j-1}\circ T_\omega x\\
&\supset  T_\omega^{-1}B_{d_{n-1}^{\theta \omega}}(T_\omega x, \epsilon)
\end{align*}
and hence $\mu_{\omega}(B_{d_{n}^{\omega}}(x, \epsilon))\geq\mu_{\omega}(T_\omega^{-1}B_{d_{n-1}^{\theta \omega}}(T_\omega x, \epsilon))=\mu_{\theta \omega}(B_{d_{n-1}^{\theta \omega}}(T_\omega x, \epsilon))$ for $\mathbb{P}$-a.e $\omega$ by using the fact $T_\omega\mu_\omega=\mu_{\theta \omega}.$ This shows   for $\mu$-a.e  $(\omega, x)$
$$F(\omega,x)=	\overline{h}_{\mu_{\omega}}^{BK}(T, x, \epsilon)\leq\overline{h}_{\mu_{\theta \omega}}^{BK}(T, T_\omega x, \epsilon)=F\circ\Theta (\omega,x).$$
Since $\mu$ is ergodic, this shows  for $\mu$-a.e  $(\omega, x)$ $\overline{h}_{\mu_{\omega}}^{BK}(T, x, \epsilon)=\overline{h}_{\mu}^{BK}(T, \epsilon)$.
\end{proof}

The  Lemma \ref{SMB} states the well-known Shannon-McMillan-Breiman Theorem for RDS \cite{Bog}. 

\begin{lem}[Shannon-McMillan-Breiman Theorem]\label{SMB} 
Let $T$ be a continuous bundle RDS over a measure-preserving  system  $(\Omega, \mc{F},\mb{P}, \theta)$. Let $\mu\in E_{\mathbb{P}}(T)$ and $\xi$  be  a finite partition  of $\mathcal{E}$. Then for $\mu$-a.e $(\omega,x)$, 
\begin{align*}
\lim\limits_{n\rightarrow \infty}-\dfrac{1}{n} \log \mu_{\omega}(A_{\xi, \omega}^{n}(x))=h_{\mu}^{\bf r}(T, \xi).
\end{align*}
\end{lem}

\begin{thm}\label{main3}
Let $T$ be a continuous bundle RDS over an ergodic measure-preserving  $(\Omega, \mc{F},\mb{P}, \theta)$. Then
\begin{align*}
&\mb{E}\umdim(T,X,d)=\limsup\limits_{\epsilon \rightarrow 0}\frac{1}{|\log \epsilon|}\sup_{\mu\in {M}_{\mathbb{P}}(T)}\overline{h}_{\mu}^{BK}(T, \epsilon),\\
&
\mb{E}\lmdim(T, X, d)=\liminf\limits_{\epsilon\rightarrow 0} \frac{1}{|\log \epsilon|}\sup_{\mu\in {M}_{\mathbb{P}}(T)}\overline{h}_{\mu}^{BK}(T, \epsilon).
\end{align*}
\end{thm}

\begin{proof}
It suffices to show the first equality.
Fix $\epsilon>0$ and $\mu\in E_{\mb{P}}(T)$. Let $\xi$ be a finite Borel partition of $X$ with ${\rm diam} \xi<\epsilon$.  Then for all $(\omega,x)$, $ A_{(\Omega\times\xi)_{\mc{E}},\omega}^{n}(x)\subset B_{d_{n}^{\omega}}(x, \epsilon)$ holds  for every $n\in \mb{N}$. This implies that
\begin{align*}
&\limsup\limits_{n \rightarrow \infty} -\frac{1}{n}\log \mu_{\omega}(B_{d_{n}^{\omega}}(x, \epsilon))\leq \lim\limits_{n\rightarrow \infty}-\frac{1}{n} \log \mu_{\omega}(A_{\Omega\times \xi,\omega}^{n}(x)).
\end{align*}
By Proposition \ref{prop 3.2} and Lemma \ref{SMB}, we have 
$$\overline{h}_{\mu}^{BK}(T,\epsilon)\leq h_{\mu}^{\bf r}(T, (\Omega\times\xi)_{\mc{E}}).$$
Therefore, 
\begin{align*}
\mb{E}\umdim(T,d)&=\limsup_{\epsilon \rightarrow 0}\frac{1}{|\log \epsilon|}\sup_{\mu \in {E}_{\mathbb{P}}(T)}\inf_{{\rm diam }\xi \leq \epsilon,\atop \xi \in \mathcal{P}_X} h_{\mu}^{\bf r}(T,(\Omega\times \xi)_{\mc{E}}),\text{by Theorem} \ref{main} \\
& \geq\limsup_{\epsilon \rightarrow 0}\frac{1}{|\log \epsilon|} \sup_{\mu \in {E}_{\mathbb{P}}(T)} \overline{h}_{\mu}^{BK}(T, \epsilon).
\end{align*} 

By (\ref{eq5}), there exists  a $\mu$-full measure  set $E\subset \Omega\times X$ so that for $(\omega,x)\in E$, 
 \begin{align*}
 	\limsup_{n \rightarrow \infty}-\frac{1}{n} \log \mu_{\omega}(B_{d_{n}^{\omega}}(x, \epsilon))=\overline{h}_{\mu}^{BK}(T, \epsilon).
 \end{align*} 
Then $\mathbb{P}(\pi_{\Omega} E)=1$ and  $\mu(E)=\int_{\pi_{\Omega} E}\mu_{\omega}(E_{\omega})d\mathbb{P}(\omega)=1,$ where $E_\omega=\{x\in \mathcal{E}_\omega:(\omega,x)\in E\}$.   So we can assume that   $\mu_{\omega}(E_{\omega})=1$ for all $\omega\in \pi_{\Omega} E$. Given $\omega\in \pi_{\Omega} E$,  $\rho>0$ and $n\in \mb{N}$, set 
 \begin{align*}
 G_{n, \rho}^{\omega}=\left\lbrace x\in \mathcal{E}_\omega: -\frac{1}{n}\log \mu_{\omega}(B_{d_{n}^{\omega}}(x, \epsilon))< \overline{h}_{\mu}^{BK}(T, \epsilon)+\rho\right\rbrace .
 \end{align*}
Let  $0<\delta<1$. Then for all sufficiently large $n\in \mb{N}$ (depending on $\delta,\omega,\rho$), one has $\mu_{\omega}(G_{n, \rho}^{\omega})>1-\delta$. Let  $H_{n}$ be a maximal $(n, 2\epsilon, \omega)$-separated subset of $G_{n, \rho}^{\omega}$.  Then  it is also an $(n, 2\epsilon, \omega)$-spanning subset of $G_{n,\rho}^{\omega}$ and  the family $\left\lbrace B_{d_{n}^{\omega}}(x, \epsilon): x\in H_{n} \right\rbrace $  is pairwise disjoint. It follows that $\mu_{\omega}(\bigcup_{x\in H_{n}} B_{d_{n}^{\omega}}(x, 2\epsilon))\geq \mu_{\omega}(G_{n, \rho}^{\omega})>1-\delta$
and 
 \begin{align*}
 \# H_{n}\cdot e^{-n(\overline{h}_{\mu}^{BK}(T, \epsilon)+\rho)}\leq \sum_{x\in H_{n}} \mu_{\omega}(B_{d_{n}^{\omega}}(x, \epsilon))=\mu_{\omega}(\bigcup_{x\in H_{n}} B_{d_{n}^{\omega}}(x, \epsilon))\leq 1.
 \end{align*} 
 Then 
 $N_{\mu_{\omega}}^{\delta}(n, 2\epsilon)\leq \#H_{n}\leq e^{n(\overline{h}_{\mu}^{BK}(T, \epsilon)+\rho)}$. This yields that
 \begin{align*}
 \overline{h}_{\mu}^{BK}(T,\epsilon)+\rho&\geq
 \int_{\pi_{\Omega} E} \limsup_{n \rightarrow \infty} \frac{1}{n} \log N_{\mu_{\omega}}^{\delta}(n, 2\epsilon) d \mb{P}(\omega)\\
 &=
 \int \limsup_{n \rightarrow \infty} \frac{1}{n} \log N_{\mu_{\omega}}^{\delta}(n, 2\epsilon) d \mb{P}(\omega)\\
 &\geq  \limsup_{n \rightarrow \infty} \frac{1}{n} \int \log N_{\mu_{\omega}}^{\delta}(n, 2\epsilon) d \mb{P}(\omega),\text{by Fatou's lemma}.
 \end{align*}
 Letting $\delta \rightarrow 0$ and then letting  $\rho \rightarrow 0$, we obtain $\overline{h}_{\mu}^{K}(T, 2\epsilon) \leq \overline{h}_{\mu}^{BK}(T, \epsilon)$  for every $\mu\in E_{\mb{P}}(T)$.  Then by Theorem \ref{main2}, we have 
 \begin{align*}
\mb{E}\umdim(T,d)=\limsup\limits_{\epsilon \rightarrow 0} \frac{1}{|\log \epsilon|}\sup_{\mu \in {E}_{\mathbb{P}}(T)} \overline{h}_{\mu}^{K}(T, \epsilon)\leq \limsup\limits_{\epsilon \rightarrow 0} \frac{1}{|\log \epsilon|}\sup_{\mu \in {E}_{\mathbb{P}}(T)}\overline{h}_{\mu}^{BK}(T, \epsilon).
 \end{align*}
\end{proof}

{\bf Problem}: Do we have 
variational principle  for metric mean dimension in terms of  $\underline{h}_{\mu}^{BK}.$

\section*{Acknowledgement} 
The first author was supported by NNSF of China (12201328),
Zhejiang Provincial Natural Science Foundation of China  (LQ22A010012) and Ningbo Natural Science Foundation (2022J145). The first and second authors were supported by NNSF of China (11671208 and 11971236). The third author is supported by Project funded by China Postdoctoral Science Foundation (2023TQ0066). We would like to express our gratitude to Tianyuan Mathematical Center in Southwest China, Sichuan University and Southwest Jiaotong
University for their support and hospitality.


\end{document}